\documentclass[11pt]{amsart}
\usepackage[margin=1in]{geometry}                
\usepackage{graphicx}
\usepackage{amssymb}
\usepackage{amsmath}
\usepackage{epstopdf}
\usepackage{amsthm}

\usepackage{hyperref}
\hypersetup{
  colorlinks=true, 
   linkcolor=red,  
}
\usepackage{algorithmic}

\usepackage{Young}

\usepackage{amsfonts,amsmath,amsthm, amssymb,amscd}

\usepackage{graphicx}

\newcommand{\ignore}[1]{ }

\newtheorem{lemma}{Lemma}
\newtheorem{theorem}{Theorem}
\newtheorem{prop}{Proposition}
\newtheorem{remark}{Remark}

\usepackage[enableskew]{youngtab}

\def\e{\mathbb{E}\, }
\def\P{\mathbb{P}}

\title[Effective bounds for Stirling numbers]{Completely effective error bounds for Stirling Numbers of the first and second kinds via Poisson Approximation}
\author{Richard Arratia and Stephen DeSalvo}
\date{November 30, 2015}                                           

\begin{document}

\begin{abstract}
  We provide completely effective error estimates for Stirling numbers of the first and second kind, denoted by $s(n,m)$ and $S(n,m)$, respectively. 
These bounds are useful for values of $m \geq  n - O(\sqrt{n})$. 
An application of our Theorem~\ref{Stirling bound} yields, for example,  
\[ s(10^{12},\ 10^{12}-2\times 10^6)/10^{35664464} \in  [ 1.87669, 1.876982 ], \]
\[ S(10^{12},\ 10^{12}-2\times 10^6)/10^{35664463} \in  [ 1.30121, 1.306975 ]. \]
The bounds are obtained via Chen-Stein Poisson approximation, using an interpretation of Stirling numbers as the number of ways of placing non-attacking rooks on a chess board.  
  
   As a corollary to Theorem~\ref{Stirling bound}, summarized in Proposition~\ref{full range}, we obtain two simple and explicit asymptotic formulas, one for each of $s(n,m)$ and $S(n,m)$, for the parametrization $m = n - t\, n^a$, $0 \leq a \leq \frac{1}{2}.$ 
These asymptotic formulas agree with the ones originally observed by Moser and Wyman in the range $0<a<\frac{1}{2}$, and they connect with a recent asymptotic expansion by Louchard for $\frac{1}{2}<a < 1$, hence filling the gap at $a = \frac{1}{2}$. 

We also provide a generalization applicable to rook and file numbers. 


\smallskip
\noindent \textbf{MSC classes:} 05A16, 60C05, 11B73

\end{abstract}

\maketitle
\tableofcontents

\section{Introduction}

  The asymptotic analysis of Stirling numbers of the first kind and of the second kind, which we denote by $s(n,m)$ and $S(n,m)$, respectively, has a history dating back at least to Jordan~\cite{Jordan}, who found asymptotic formulas for $s(n,m)$ and $S(n,m)$ when $m$ is fixed and $n$ tends to infinity.  Since then, a number of asymptotic formulas covering a range of parameter values when both $n$ and $m$ tend to infinity has been provided by numerous authors.  
In a series of two papers, Moser and Wyman \cite{MoserWyman1, MoserWyman2} treated the case when $0<m<n$ and $\lim n-m = \infty$ by a saddle point analysis; see also~\cite{GoodMultinomial, GoodAsymptotic}.  
More recently, Chelluri, Richmond, and Temme \cite{CRT} succinctly summarized the asymptotic analysis of Stirling numbers of both kinds for the full range of parameter values.  Their approach is the same as Moser and Wyman's, but generalized to include real values of the parameters.  

  For Stirling numbers of the second kind, Moser and Wyman \cite{MoserWyman2} obtained a complete asymptotic expansion with an explicit hard error term valid for parameter values $m = n - o(\sqrt{n})$, namely,
\begin{eqnarray*}
S(n,m) & = &  \binom{n}{m}q^{-(n-m)} \sum_{k=0}^{n-m-1} A_k^{n-m} q^k, \\
\end{eqnarray*}
where $q = \frac{2}{n-m}$ and $A_k^r$ is a polynomial in $r$ for each $k=1,2,\ldots$, with
\begin{eqnarray*}
\sum_{k=0}^{n-m-1} A_k^{n-m} q^k & = &\left[1+\frac{(n-m)_2}{12}q + \frac{(n-m)_4q^2}{288}+\left(\frac{(n-m)_6}{10368}-\frac{(n-m)_4}{1440}\right)q^3 + \cdots \right]. \\
\end{eqnarray*}
In particular, using only the first $s$ powers of $q$ in the summation above, they obtained
\begin{equation}\label{error}
S(n,m)  =   \binom{n}{m}q^{-(n-m)} \left( \sum_{k=0}^{s} A_k^{n-m} q^k + E_s \right), \qquad |E_s| \leq \frac{(2(n-m)^2/5m)^{s+1}}{1-(2(n-m)^2/5m)}.
\end{equation}
This error tends to zero for values of $m = n - o(\sqrt{n})$, and does provide completely effective error estimates in that range.  
 
Their analysis continued in general for all $0<m<n$ such that $\lim_{n\to\infty} n - m = \infty$, although their resulting formula is defined in terms of implicit parameters, and no hard error bounds were obtained.  
 The first term in the asymptotic expansion is given by 
\begin{equation}\label{implicit R}
S(n,m) =  \frac{n! \left(\exp(R)-1\right)^m}{2 R^n m!\sqrt{\pi m R H}}(1+ O(m^{-1/3})), 
\end{equation}
where $R$ is the solution to 
\[ R(1-e^{-R})^{-1} = n/m, \]
and
\[ H = e^{R}(e^{R}-1-R)/2(e^R-1)^2. \]
Quoting from \cite[Page 40]{MoserWyman2}\footnote{but using our equation numbers.},
\begin{quotation}
In order to complete our discussion it would be valuable to give an accurate estimate of the accuracy of our formulae.  Such an estimate would involve the order of magnitude of the terms that have been dropped in the derivation of our asymptotic formula \eqref{implicit R}.  These terms are of a very complicated nature, and such an estimate would be difficult to obtain.
\end{quotation}
They go on to provide a heuristic analysis of the error and demonstrate its accuracy for particular values of the parameters.  While giving the reader a general sense of the accuracy for large values of $n$, there are no hard error bounds akin to Equation~\eqref{error}.

For Stirling numbers of the first kind, similarly, we have from \cite{MoserWyman1}
\begin{equation}\label{implicit R s1}
|s(n,m)| = \frac{\Gamma(n+R)}{R^m\Gamma(R)\sqrt{2\pi H}}(1 + O(m^{-1})),
\end{equation}
where $R$ is the unique solution to 
\[ \sum_{k=0}^{n-1}\frac{R}{R+k} = m, \]
and
\[ H = m - \sum_{k=0}^{n-1}\frac{R^2}{(R+k)^2}. \]
Quoting from \cite[Pages 142--143]{MoserWyman1},
\begin{quotation}
One of the defects of formula \eqref{implicit R s1} is the fact that we have given no estimate of the error involved in using only those terms shown in \eqref{implicit R s1}.  ... Unfortunately, we have been unable to give even such a crude estimate of the error involved in the use of \eqref{implicit R s1}.
\end{quotation}

Our main result, Theorem~\ref{Stirling bound},  is an explicit upper and lower bound for each of the Stirling numbers of the first and second kind that is valid for all finite values of $n$ and $m$.  
Our theorem is an application of Poisson approximation via the Chen-Stein method, which provides completely effective error estimates by bounding the total variation distance between the laws of a sum of Bernoulli random variables and an independent Poisson random variable with the same mean.  
  We use an interpretation of the Stirling numbers as the number of ways of placing $k$ non-attacking rooks on the strictly lower triangular half of an $n~\times~n$ chess board \cite{KaplanskyRiordan}, see also \cite[Page 75]{Stanley}.  For Stirling numbers of the second kind, this corresponds to the value $S(n, n-k)$.  A similar interpretation for Stirling numbers of the first kind also holds, with value $|s(n, n-k)|$, except one takes ``non-attacking" to mean only column-wise.
 
 In Section~\ref{Stirling} we recall the relevant definitions and a summary of asymptotic formulas for various ranges of the parameters given in \cite{CRT}.  A more complete historical treatment is contained in \cite{CRT, Louchard2}, and we refer the interested reader to the references therein. 
 We state in Proposition~\ref{full range} two simple, explicit asymptotic formulas for the Stirling numbers, one for each kind, valid for values 
$m = n-t\, n^a$, $0\leq a \leq \frac{1}{2}$, $t>0$, $n\to\infty$, which agrees with 
Moser and Wyman's analysis \cite{MoserWyman1, MoserWyman2} for $0<a<\frac{1}{2}$. 
The asymptotic formulas involving implicitly defined parameters were originally derived by Moser and Wyman \cite{MoserWyman1, MoserWyman2} in the forms of Equation~\eqref{implicit R} and Equation~\eqref{implicit R s1}, without hard error bounds, and in Theorem~\ref{Stirling bound} we state them in a compact, explicit formula with completely effective error estimates that hold for all finite values of the parameters. 
 
Section~\ref{main result} contains the statement of Theorem~\ref{Stirling bound} and the Poisson approximation theorem used to obtain it.
 In Section~\ref{proofs} we prove our main result, the concrete error bounds on the Stirling numbers, by defining an appropriate collection of indicator random variables and obtaining a bound in total variation distance of the law of their sum to an independent Poisson random variable with the same mean.  
  Our method does not involve a single integral, but rather constructs a coupling between two random variables which consequently requires only a straightforward counting argument.   
 In Section~\ref{remarks} we explore a variation of the Poisson approximation approach which gives alternative bounds, and state a generalization which applies to more general board shapes. 
 
\section{Stirling numbers}
\label{Stirling}
The Stirling numbers of the first and second kind, denoted by $s(n,m)$ and $S(n,m)$, respectively, where $n \geq m \geq 1$ are integers, are defined as follows:  let \[(x)_n = x(x-1)\ldots (x-n+1)\] denote the falling factorial function, then we have
\begin{equation}
\sum_{m=0}^n s(n,m) x^m = (x)_n,
\end{equation}

\begin{equation}
\sum_{m=0}^n S(n,m) (x)_m = x^n.
\end{equation}
Alternatively, one can define the Stirling numbers by the recursions:
\begin{align*}
 s(n+1, m) = -n\, s(n,m)+ s(n,m-1), & \qquad n \geq m  \geq 1, \\
 s(0,0) = 1,\ s(n,0) = s(0,n) = 0, &
 \end{align*}
and
\begin{align*}
S(n+1, m) = m\, S(n,m)+ S(n, m-1), & \qquad  n \geq m \geq 1, \\
S(0,0) = 1,\ S(n,0) = S(0,n) = 0. &  
\end{align*}

The values for $m = 1, 2$ and all $n \geq m$ have simple formulas:
\[
\begin{array}{ll}
s(n,1) = (-1)^{n-1}(n-1)!, & s(n,2) = (-1)^{n}(n-1)! \left(1+\frac{1}{2} + \frac{1}{3} + \ldots \frac{1}{n}\right). \\ \\
 S(n,1) = 1, & S(n,2) = 2^{n-1}-1. 
\end{array}
\]

The numbers $s(n,m)$ are sometimes referred to as the signed Stirling numbers of the first kind, since they take both positive and negative values.  If, however, one considers instead $|s(n,m)|$, i.e., the unsigned Stirling numbers of the first kind, then these numbers count the number of permutations of~$n$ distinct elements into exactly~$m$ disjoint cycles.  Hence,
\[ n! = \sum_{m=0}^n |s(n,m)|. \]
The simple relation $s(n,m) = (-1)^{n-m} |s(n,m)|$ allows us to state our results in terms of the more natural combinatorial quantity $|s(n,m)|$.  

The numbers $S(n,m)$ count the number of set partitions of size~$n$ into exactly~$m$ non-empty blocks.  Hence,
\[ B_n = \sum_{m=0}^n S(n,m), \]
where $B_n$ is the $n$-th Bell number, which counts the number of partitions of a set of size~$n$.  

An alternative characterization of Stirling numbers, the one we utilize in the present paper, involves non-attacking rooks on a chess board.

\begin{theorem} Let $B$ denote the set of strictly lower triangular squares of an $n \times n$ chess board.  
\begin{enumerate}
\item[a.] \cite{KaplanskyRiordan} The number $S(n, m)$ counts the number of ways of placing $n-m$ rooks on $B$ such that no pairs of rooks are in the same row or column.  
\item[b.] \cite{GarsiaRemmel} The number $|s(n, m)|$ counts the number of ways of placing $n-m$ rooks on $B$ such that no pairs of rooks lie in the same column.
\end{enumerate}
\end{theorem}

We briefly summarize a few of the known asymptotic results for Stirling numbers. 
For a more comprehensive treatment, see~\cite{CRT, Louchard2}.  References are given in the right-most column.

\begin{theorem}
 In what follows, $\delta$ is a positive constant that stays fixed as $n\to\infty$.  
For Stirling numbers of the first kind, as $n\to\infty$, 
\[
\begin{array}{llll}
|s(n,m)| = & \frac{(n-1)!}{(m-1)!}\left( \log n + \gamma\right)^{m-1}(1+O(\log^{-1} n)), & m =O(\log(n)) , & \text{\cite{MoserWyman1, Jordan, Hwang}}, \\ \\
|s(n,m)| = & \frac{\Gamma(n+R)}{R^m\Gamma(R)\sqrt{2\pi H}}(1 + O(n^{-1})), & \sqrt{\log n} \leq m \leq n - n^{1/3}, & \text{\cite{MoserWyman1, CRT}}, \\ \\
|s(n,m)| = & \binom{n}{m}\left(\frac{m}{2}\right)^{n-m}\left(1+O(n^{-1/3})\right), & n-n^{1/3}\leq m \leq n, & \text{\cite{MoserWyman1, CRT}}, 
\end{array}
\]
where $\gamma$ is Euler's constant, $R$ is the unique solution to 
\[ \sum_{\ell=0}^{n-1}\frac{R}{R+\ell} = m, \]
and
\[ H = m - \sum_{\ell=0}^{n-1}\frac{R^2}{(R+\ell)^2}. \]

For Stirling numbers of the second kind, as $n\to\infty$, 
\[
\begin{array}{llll}
S(n,m) = & \frac{m^n}{m!} \exp\left[ \left(\frac{n}{m}-m\right)e^{-n/m}\right](1+o(1)), & m < n / \ln n,  & \text{\cite[Page 144]{Sachkov}}, \\ \\
S(n,m) = & \frac{n! \left(\exp(R)-1\right)^m}{2 R^n m!\sqrt{\pi m R H}}(1+ O(n^{-1})), & 0<\delta \leq m \leq n-n^{1/3}, & \text{\cite{MoserWyman1, CRT} }, \\ \\
S(n,m) = & \frac{n^{2(n-m)}}{2^{n-m} (n-m)!}(1+O(n^{-1/3})) & n-n^{1/3} \leq m \leq n , & \text{\cite{Hsu1948, CRT}}, \\ \\
\end{array}
\]
where $R$ is the solution to 
\[ R(1-e^{-R})^{-1} = n/m, \]
and
\[ H = e^{R}(e^{R}-1-R)/2(e^R-1)^2. \]
\end{theorem}

These results in a sense completely describe the asymptotic behavior of the Stirling numbers of both kinds for various values of $n$ and $m$.  Our contribution is the addition of error estimates, and an explicit formula in terms of $n$ and $m$.  
  In fact, recently, in \cite{Louchard1, Louchard2}, the implicit dependence on $R$ in the formulas above was explicitly calculated for $m = n-n^a$, $a>1/2$, and an asymptotic expansion was given that depends explicitly on the value of $a$.  These results are summarized below.  

\begin{theorem}[\cite{Louchard1, Louchard2}]\label{Louchard Theorem}
Fix any $a \in (\frac{1}{2},1)$.  Then as $n$ tends to infinity we have
\[ |s(n, n-n^a)| \sim \frac{e^{T_1}}{\sqrt{2\pi\, n^{a}}}, \]
\[ S(n, n-n^a) \sim \frac{e^{T_1'}}{\sqrt{2\pi\,n^{a}}}, \] 
where
\begin{equation}\label{T1}
 T_1 = x\left[1 - \ln(2) + 2 \ln(y) +\ln(x) - \frac{2}{3y} - \frac{2}{9y^2} + \ldots\right],
 \end{equation}
 \begin{equation} \label{T1p}
 T_1' = x\left[1 - \ln(2) + 2 \ln(y) +\ln(x) - \frac{4}{3y} - \frac{5}{9y^2} + \ldots\right],
\end{equation}
and $x = n^a$, $y = n^{1-a}$.  
\end{theorem}  

 Theorem~\ref{Louchard Theorem} above is from a multiseries expansion (see e.g. \cite{Shackell}) and requires $n^{1-a} = o(n^a)$ and $n^a = o(n)$, hence the case $a=1/2$ does not apply directly.  In addition, the explicit error term in Equation~\eqref{error} does not go to zero for $a = \frac{1}{2}$.  Nevertheless, we obtain as a corollary to our main result that the coefficient of $x/y$ in the asymptotic expansions of $T_1$ and $T_1'$ in Equations~\eqref{T1}~and~\eqref{T1p} indeed give the correct term when $a = 1/2$. 

\begin{prop} \label{full range} Suppose $m = n - O(\sqrt{n})$, and let 
\[ \mu := \frac{\binom{n-m}{2}{\binom{n}{3}}}{\binom{\binom{n}{2}}{2}}. \]
Then asymptotically as $n\to\infty$, 
\begin{equation}\label{s1 full}
|s(n,m)| \sim \binom{\binom{n}{2}}{n-m}e^{-2\mu}\left(1 + O\left(\frac{n-m}{n}\right) \right),
\end{equation}
\begin{equation}\label{S2 full}
S(n,m) \sim \binom{\binom{n}{2}}{n-m}e^{-\mu}\left(1 + O\left(\frac{n-m}{n}\right) \right). 
\end{equation}
In particular, let $t>0$ and $0\leq a\leq \frac{1}{2}$ be fixed.  Then for $m = n - t\, n^a$, as $n \to\infty$ we have 
\begin{equation}\label{rootn s}
|s(n,m)| \sim \binom{\binom{n}{2}}{t\, n^a} e^{ - \frac{2}{3} t^2 n^{2a-1}} \sim \frac{1}{\sqrt{2\pi\, t^5\, n^a}}e^{t\, n^a \left((2-a)\ln n + 1 - \ln(2)\right) - \frac{2}{3} t^2n^{2a-1}},
\end{equation}
\begin{equation}\label{rootn S}
S(n,m) \sim \binom{\binom{n}{2}}{t\, n^a} e^{ - \frac{4}{3} t^2 n^{2a-1}}\sim \frac{1}{\sqrt{2\pi\, t^5\, n^a}}e^{t\, n^a \left((2-a)\ln n + 1 - \ln(2)\right) - \frac{4}{3} t^2n^{2a-1}}.
\end{equation}

\end{prop}

\begin{proof}

Equations~\eqref{s1 full} and~\eqref{S2 full} follow from Theorem~\ref{Stirling bound}.  
\end{proof}

\section{Main Result} \label{main result}

We begin by defining total variation distance between the laws of two random variables as\footnote{Note that our definition of total variation distance differs from another conventional version by a factor of 2.  This is so that the total variation distance is a number between 0 and 1 rather than between 0 and 2.}
\begin{equation}\label{dtv}
 d_{TV}(\mathcal{L}(X), \mathcal{L}(Y)) = \sup_{A\subset \mathbb{R}} | P(X \in A) - P(Y \in A)|. 
 \end{equation}
In particular, when $X$ and $Y$ are discrete random variables we have
\[ d_{TV}(\mathcal{L}(X), \mathcal{L}(Y)) =  \frac{1}{2}\sum_i |P(X=i) - P(Y=i)|. \]
In order to simplify notation, we make the common abuse of notation by denoting the left-hand side of the above equation as $d_{TV}(X,Y)$.  
The specialization of Equation~\eqref{dtv} with $A = \{0\}$ is of particular interest; that is,
\[ |P(X=0) - P( Y=0) | \leq d_{TV}(X,Y). \]

Now we recall a classical Stein's method result on Poisson convergence due to Chen, see for example \cite[Page 130]{Grimmett}.  
\begin{theorem}[Chen-Stein]\label{ChenSteinThm}
For an index set $\Gamma$, suppose that $W := \sum_{\alpha \in \Gamma} X_\alpha$  is a sum of indicator random variables, 
and for each $\alpha\in\Gamma$, let  $V_{\alpha}$ be a random variable on the same probability space as W, with distribution 
\begin{equation}\label{Stein Distribution}
 \mathcal{L}(V_\alpha) =  \mathcal{L}(W-1| X_\alpha = 1). \end{equation}
Define $p_\alpha := \e X_\alpha$ and $\lambda := \sum_{\alpha} p_\alpha \in (0,\infty)$, and let $P$ denote an independent Poisson random variable with expectation $\lambda$.
Then we have
\begin{equation}\label{ChenSteinBound}
d_{TV}(W, P) \leq \min(1,\lambda^{-1}) \sum_{\alpha \in \Gamma} p_\alpha \e |W - V_\alpha|.
\end{equation}
\end{theorem}

In our setting, we take the board $B$ to be the strictly lower triangular squares of an $n \times n$ chess board.
We define $N := |B| = \binom{n}{2}$, and we place $k$ rooks in distinct squares, with all $\binom{N}{k}$ placements equally likely.  
We label  the  rooks by the numbers in the set $\{1, 2, \ldots, k\}$, and we index the unordered pairs of rooks by $\alpha = \{a,b\}$, where $1 \leq a < b \leq k$.
Then we define $X_\alpha$  to be the indicator random variable that takes the value 1 when the pair of rooks $\{a,b\}$ are attacking, so that  $W := \sum_\alpha X_\alpha$ is the total number of pairs of rooks that are attacking.
Theorem~\ref{ChenSteinThm} can be used to show that the random variable $W$ is approximately Poisson distributed, with mean $\e W = \sum_\alpha \e X_\alpha$.
The event of interest to us is $\{W = 0\}$, since we have
\begin{equation}\label{W=0}
 P(W = 0) = \frac{S(n,n-k)}{\displaystyle \binom{N}{k}},
 \end{equation}
using the standard definition of rooks, and
\begin{equation}\label{W=0 column}
 P(W = 0) = \frac{|s(n,n-k)|}{\displaystyle \binom{N}{k}},
 \end{equation}
when rooks only attack column-wise.  

Let $Y$ denote a Poisson random variable with expectation $\lambda = \sum_{\alpha} \e X_\alpha$, then we have, using the standard definition of rooks, for each $n \geq 3$ and $n \geq k  \geq 2$, 
\begin{equation}\label{P 0}
 |P(W = 0) - P(Y = 0)| = \left|\frac{S(n,n-k)}{\displaystyle \binom{\binom{n}{2}}{k}} - e^{-\lambda} \right| \leq d_{TV}(W,Y).
 \end{equation}
One can then bound the right-hand side of Equation~\eqref{P 0} using Theorem~\ref{ChenSteinThm}.  The calculation for Stirling numbers of the first kind is analogous.  

\begin{theorem}
\label{Stirling bound}
Suppose $n\geq 3$ and $n \geq k \geq 2$.  Let $N := \binom{n}{2},$ and define
\[ \mu \equiv \mu_{n,k} := \frac{\binom{k}{2}\binom{n}{3}}{ \binom{N}{2}},
\]
\[ d_1 := \frac{\binom{n}{3}}{\binom{N}{2}} + \frac{13-12 k+3 k^2}{\binom{N}{2}}  + \left(1 - \frac{13-12 k+3 k^2}{\binom{N}{2}}\right)\left(\frac{4 (k-2) \binom{n}{3}}{\binom{N}{2}} + \frac{6 \binom{n}{4}(k-2)}{\binom{n}{3} \left(N-2\right)}\right), \]
\[ d_2 := \frac{2\binom{n}{3}}{\binom{N}{2}} + \frac{13-12 k+3 k^2}{\binom{N}{2}} + \left(1 - \frac{13-12 k+3 k^2}{\binom{N}{2}}\right)\left(\frac{8 (k-2) \binom{n}{3}}{\binom{N}{2}} + \frac{6 \binom{n}{4}(k-2)}{\binom{n}{3} \left(N-2\right)} + \frac{(k-2)}{\left(N-2\right)} \left(\frac{5n-11}{4}\right)\right), \]
\[ D_{n,k}^{(1)} := \min\left(d_1, \mu\, d_1, 1\right), \]
\[ D_{n,k}^{(2)} := \min\left(d_2, 2\, \mu\, d_2, 1\right). \]
Then we have
\[\binom{N}{k} e^{-2\mu}\left( 1 - e^{2\mu} D_{n,k}^{(2)}\right) \leq S(n, n-k) \leq \binom{N}{k} e^{-2\mu}\left( 1 + e^{2\mu} D_{n,k}^{(2)}\right),\]
\[\binom{N}{k} e^{-\mu}\left( 1 - e^\mu D_{n,k}^{(1)}\right) \leq |s(n, n-k)| \leq \binom{N}{k} e^{-\mu}\left( 1 + e^\mu D_{n,k}^{(1)}\right).\]
\end{theorem}

Note that the error term is absolute for all finite values of $n\geq 3$ and $n \geq k \geq 2$, and goes to 0 for $k = O(\sqrt{n})$, which proves Proposition~\ref{full range} for values of $0\leq a \leq \frac{1}{2}$.  
We note that this range of parameter values is strictly greater than the range in Equation~\eqref{error} for Stirling numbers of the second kind, and the first of its kind for Stirling numbers of the first kind. 

\section{Proof of Theorem~\ref{Stirling bound}}

\label{proofs}
We recall our setting: we take the board $B$ to be the strictly lower triangular squares of an $n \times n$ chess board.
We define $N := |B| = \binom{n}{2}$, and we place $k$ rooks in distinct squares, with all $\binom{N}{k}$ placements equally likely.  We index each of the $k$ rooks by numbers in the set $\{1, 2, \ldots, k\}$, and index each pair of rooks by $\alpha = \{a,b\}$, where $1 \leq a < b \leq k$.  Then $X_\alpha$ is defined as the indicator random variable that takes the value 1 when the pair of rooks $\{a,b\}$ are attacking.  We let $W = \sum_\alpha X_\alpha$ equal the total number of pairs of rooks that are attacking. 

We define indicator random variables $R_\alpha$ and $C_\alpha$, where 
\[ R_\alpha := 1(\text{rooks $a$ and $b$ share a row}), \]
and
\[ C_\alpha := 1(\text{rooks $a$ and $b$ share a column}). \]
For Stirling numbers of the first kind, we have $X_\alpha = C_\alpha$, and for Stirling numbers of the second kind, we have $X_\alpha = R_\alpha + C_\alpha$.  We shall take the latter for the rest of the proof, since the bound for Stirling numbers of the first kind can be obtained by a subset of the calculated quantities for the bound on Stirling numbers of the second kind.  We define $r_\alpha := \e R_\alpha$ and $c_\alpha := \e C_\alpha$.  

First we establish the Poisson rate.  We shall use often the identity
\[ \sum_{\ell=b}^{a-1} \binom{a}{\ell} = \binom{a}{b+1}, \qquad a > b \geq 1. \]
We have 
\begin{equation}
r_\alpha + c_\alpha =  \e (R_\alpha) + \e(C_\alpha) = \frac{1}{\binom{N}{2}}\left(\sum_{r=2}^{n-1} \binom{r}{2}  + \sum_{c=1}^{n-2} \binom{n-c}{2}\right) = \frac{2 \binom{n}{3}}{\binom{N}{2}}.
\end{equation}
Summing over all \emph{unordered} pairs of rooks, we have
\begin{equation}\label{q sum}
\sum_{\alpha}\left( r_\alpha+c_\alpha \right) = \frac{2 \binom{k}{2}\binom{n}{3}}{ \binom{N}{2}},
\end{equation}
and so we define 
\[ \mu:= \sum_{\alpha} c_\alpha =  \frac{\binom{k}{2}\binom{n}{3}}{ \binom{N}{2}}, \]
which is the Poisson rate for Stirling numbers of the first kind, and $2\mu$ is the Poisson rate for Stirling numbers of the second kind.  

In order to proceed with the total variation bound implied by Theorem~\ref{ChenSteinThm}, we must first define a joint probability space for random variables $W$ and $V_\alpha$.  The total variation bound holds for all random variables $V_\alpha$ with distribution given by \eqref{Stein Distribution}, and we are free to choose any coupling; ideally, one that minimizes $\e |W-V_\alpha|$.  By symmetry in the rooks $1$ to $k$, in  Equation~\eqref{ChenSteinBound} we know that the value of $p_\alpha$
does not vary with $\alpha$, and  \emph{we can arrange} couplings of the $V_\alpha$ with $W$  so that $E | W - V_\alpha |$  also does not vary with $\alpha$.   Therefore we consider, from now on, and without loss of generality, the specific case  $\alpha=\{1, 2\}$.

To describe the coupling, we think of the board $B$ as containing $k$ rooks with labels $\{1,2,\ldots, k\}$ that are visible and $N-k$ rooks with labels $\{k+1, k+2, \ldots, N\}$ that are invisible.  We identify the board  $B$ with the set of integers  $\{1,2,\ldots,N\}$,  so that  $ \pi \in S_B$, a permutation of $B$, corresponds to arranging all rooks, visible and invisible, on distinct squares of the board.  Finally, let  $\pi$ be uniformly distributed over $S_B$.

The coupling is: independent of $\pi$, pick a pair of locations $(I,J)$ according to the distribution $P((I,J) = (i,j)) = p_{i,j}$ that assigns equal probability to every pair of locations on $B$ for two attacking rooks\footnote{Note that $I$ and $J$ have the same marginal distribution, which is \emph{not} uniform over $B$.}.  Let $\pi_1$ and $\pi_2$ denote the locations of rooks 1 and 2, respectively, under the permutation $\pi$.  Then swap the coordinates of rooks labeled 1 and 2 with the rooks in locations $I$ and $J$, respectively,
to produce a new permutation, which we call $\pi'$.\footnote{Formally,
with the notation $(i\ j)$ to denote a two cycle if $i \ne j$, and the identity permutation if $i=j$,
\begin{align*}
\pi^\ast & := (\pi_2\ J) \circ \pi, (\text{so that rook 2 lands at $J$}),\\
\pi' & = (\pi_1^\ast\ I) \circ \pi^\ast, (\text{so that rook 1 lands at $I$}).
\end{align*}
}

  In the arrangement of rooks given by $\pi'$, rooks 1 and 2 are at board locations $I$ and $J$, where they are attacking.  It is clear that $\pi'$ has the desired distribution, namely,~($\pi$,~conditional~on~$X_\alpha=1$);  in detail, this desired distribution is the \emph{mixture},  governed by $(p_{i,j})$, of the distributions of $\pi$ conditional on placing rook 1 at $I$ and rook 2 at $J$, under which all  $(N-2)!$ possible permutations are equally likely.

\begin{remark}\label{notation}{\rm
Although the random variable $W$ is, by definition, a function on $\Omega$, $W$ depends on $\omega \in \Omega$ only through the permutation $\pi$, and henceforth in addition to writing
\[ W(\omega)  :=  \sum_\beta X_\beta(\omega), \]
we abuse notation,  and consider  $W$ as a function on $S_B$;  taking  $\beta = \{a,b\}$ we write
\[ W(\pi) = \sum_\beta X_\beta(\pi) =\sum_{1 \le a < b \le k}   1\text{(locations $\pi_a$   and $\pi_b$   are attacking).} \]
Our coupling of $W$ with $V_\alpha$  is carried out by specifying a permutation $\pi'$, and the value of $V_\alpha$ depends on $\omega \in \Omega$ only through the permutation $\pi'$. Under the corresponding abuse of notation,  $V_\alpha$ is the same function of the permutation $\pi'$  that  $W$ was of the permutation $\pi$, except that the term corresponding to $X_\alpha$ is omitted;   that is,
\[ V_\alpha(\pi')  =  \sum_{\beta \ne \alpha} X_\beta(\pi') = \sum_{1 \le a < b \le k, (a,b) \ne \alpha}  1\text{(locations  $\pi'_a$  and $\pi'_b$   are attacking).} \]}
   \end{remark}

   \begin{remark}\label{I is J}{\rm
   If $\alpha  = \{1,2\}$  and a permutation  $\sigma \in S_B$   is such that  locations   $\sigma_1$  and  $\sigma_2$  are attacking, then
\[ V_\alpha( \sigma  )   =  V_\alpha (  \sigma \circ ( 1\ 2 )  ), \]
i.e., if rooks 1 and 2 are attacking under $\sigma$,  then interchanging their positions  leads to zero net change in $V_\alpha$.
 } \end{remark}

A consequence of Remark \ref{I is J} is that, instead of requiring the swap to result in $(\pi'_1,\pi'_2)=(I,J)$, we could have  carried out the swap allowing $\{\pi'_1,\pi'_2\}=\{I,J\}$.

 We organize our analysis of the swap into three cases, depending on the size of the overlap between the sets $\{\pi_1, \pi_2\}$ and $\{I,J\}$.
\begin{itemize}
\item[Case 1:] Both rooks lie in the two squares (in any order), i.e., $\{\pi_1, \pi_2\} = \{I,J\}$.
\item[Case 2:] One rook is in $\{I,J\}$, the other is not, i.e., $|\{\pi_1, \pi_2\} \cap \{I,J\}| = 1$.
\item[Case 3:] Neither of the two rooks is in $\{I,J\}$, i.e., $\{\pi_1, \pi_2\} \cap \{I,J\} = \emptyset$.
\end{itemize}

 Next we consider the possible values of $|W-V_\alpha|$, which depend on whether swapped rooks were visible or invisible.  

\begin{itemize}
\item[Case 1:] Rooks 1 and 2 were already attacking, in locations $\{I,J\}$, therefore $|W-V_\alpha| = X_\alpha =  1$.
\item[Case 2:] WLOG by Remark~\ref{I is J}, assume $\pi_1=I$ and we are swapping rook 2 into location $J$. 

\item[Case 2a:] Assume there was a visible rook in $J$, so when we swap there are no changes in the set of positions occupied by the $k$ visible rooks, therefore $|W-V_\alpha| = X_\alpha = 1$.
\item[Case 2b:] WLOG again by Remark~\ref{I is J}, assume we are swapping rook 2 into location $J$.  Assume there was an \emph{invisible} rook in $J$, so the swap amounts to moving rook 2, and no other rook.  We will upper bound the expected contribution from $|W-V_\alpha|$ later, using the expected gross number of changes, see Equation~\eqref{gross}.
\item[Case 3a:] We assume further that there were two \emph{visible} rooks in locations $I$ and $J$.  As in case 2a,  $|W-V_\alpha| = X_\alpha = 1$.
\item[Case 3b:] We assume that there was one visible rook and one invisible rook in locations $I$ and $J$.  This case is similar to case 2b.
\item[Case 3c:] We assume that there were two \emph{invisible} rooks in locations $I$ and $J$.  In this case, we expect twice the gross number  of changes, compared with cases 2b and 3b.
\end{itemize}

\begin{figure}
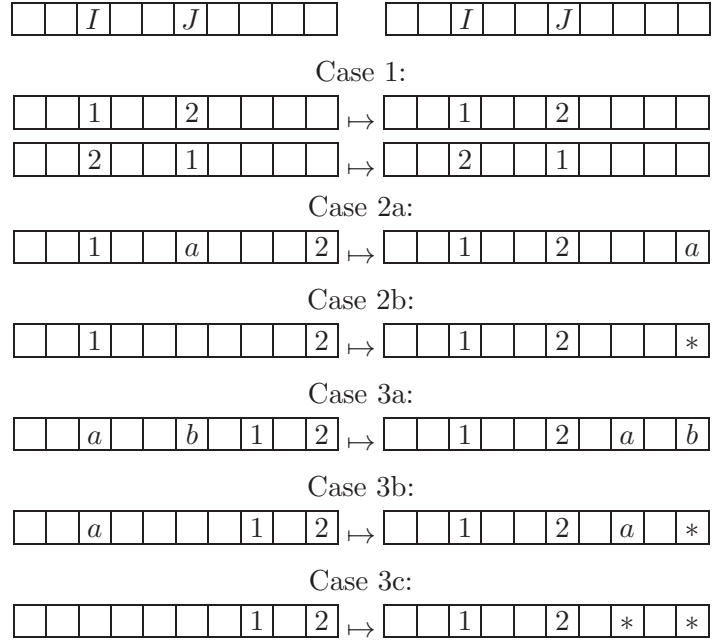
\label{Post Coupling Cases}
\[
\young(\ \ I\ \ J\ \ \ \ )  \ \ \ \ \  \young(\ \ I\ \ J\ \ \ \ )
\]
Case 1:
\[
\young(\ \ 1\ \ 2\ \ \ \ ) \mapsto \young(\ \ 1\ \ 2\ \ \ \ )
\]
\[
\young(\ \ 2\ \ 1\ \ \ \ ) \mapsto \young(\ \ 2\ \ 1\ \ \ \ ) 
\]
Case 2a:
\[
\young(\ \ 1\ \ a\ \ \ 2) \mapsto \young(\ \ 1\ \ 2\ \ \ a)
\]
Case 2b:
\[
\young(\ \ 1\ \ \ \ \ \ 2) \mapsto \young(\ \ 1\ \ 2\ \ \ \ast)
\]
Case 3a:
\[
\young(\ \ a\ \ b\ 1\ 2) \mapsto \young(\ \ 1\ \ 2\ a\ b)
\]
Case 3b:
\[
\young(\ \ a\ \ \ \ 1\ 2) \mapsto \young(\ \ 1\ \ 2\ a\ \ast)
\]
Case 3c:
\[
\young(\ \ \ \ \ \ \ 1\ 2) \mapsto \young(\ \ 1\ \ 2\ \ast\ \ast)
\]
\caption{Before and after positions of rooks $1$ and $2$ on the board relative to $I$ and $J$ and visible rooks $a$ and $b$.  A $\ast$ represents a location where a change in the number of conflicts can occur.}
\label{Coupling Cases}
\end{figure}

First, the contribution to the expectation $\e |W-V_\alpha|$ for Cases 1, 2a, and 3a are easy to calculate, since these subcases correspond to 
naming the set of two visible rooks,   chosen from  $\{1,2\}$ and $\{3,4,\ldots,k\}$, and then requiring that $\pi$ places this set of rooks in the set of positions given by $\{I,J\}.$  We have 
\begin{align*}
   P(\text{Case 1})   = &  \frac{1}{\binom{N}{2}}, \\
 P(\text{Case 2a}) = &  \frac{3(k-2)}{\binom{N}{2}}, \\
 P(\text{Case 3a}) = &  \frac{3(k-2)(k-3)}{\binom{N}{2}}.
\end{align*}
The total probability from these cases is given by 
\begin{equation}\label{cases 123} 
T := \frac{13-12 k+3 k^2}{\binom{N}{2}}.  
\end{equation}

Rather than always evaluate $\e |W - V_\alpha|$ directly, which becomes considerably more complicated in the other cases, we instead work with an upper bound.  We shall point out the precise form of each error term, so the tradeoff is apparent, and calculate Big-O estimates in Remark~\ref{remark error}.   Remark~\ref{notation} yields 
\[   W - V_\alpha = X_\alpha(\pi) + \sum_{\beta \ne \alpha} (X_\beta(\pi) - X_\beta(\pi')).
\]
Then we have the elementary upper bound
\begin{equation}\label{gross}
 |W - V_\alpha|  \le   X_\alpha(\pi) + \sum_{\beta \ne \alpha} |X_\beta(\pi) - X_\beta(\pi')|.
\end{equation}
We call the right-hand side of Equation~\eqref{gross} the \emph{gross} number of changes.

Our first upper bound of $\e |W - V_\alpha|$ is obtained via Equation~\eqref{gross}, which introduces two forms of error that we quantify below.
\begin{lemma}
\begin{equation}\label{E gross}
\e |W - V_\alpha|  =   \e X_\alpha(\pi) + \e \sum_{\beta \ne \alpha} |X_\beta(\pi) - X_\beta(\pi')| - A_1 - A_2,
\end{equation}
where 
\[A_1 = 2\,\e X_\alpha(\pi)\, \P\left(\sum_\beta X_\beta(\pi) < \sum_\beta X_\beta(\pi')\right),
\]
\[ A_2 = 2\, \e \min\left( \sum_\beta (X_\beta(\pi) - X_\beta(\pi'))^-,\ \sum_\beta (X_\beta(\pi) - X_\beta(\pi'))^+\right).
\]
\end{lemma}

\begin{proof}

\[ |W - V_\alpha| =  \begin{cases}
\ \ \, X_\alpha(\pi) + \sum_{\beta \ne \alpha} (X_\beta(\pi) - X_\beta(\pi')) , & \sum_\beta X_\beta(\pi) \geq X_\beta(\pi'), \\
-X_\alpha(\pi) + \sum_{\beta \ne \alpha} (X_\beta(\pi') - X_\beta(\pi)), & \sum_\beta X_\beta(\pi) < \sum_\beta X_\beta(\pi'). \end{cases}
\]
Taking expectation, we have 
\[ \e |W - V_\alpha| = \e \left|\sum_\beta (X_\beta(\pi) - X_\beta(\pi'))\right| + \e X_\alpha(\pi)  - 2\, \e X_\alpha(\pi)\, \P\left(\sum_\beta X_\beta(\pi) < \sum_\beta X_\beta(\pi')\right). \]
Next, we distribute the absolute value over the summation terms.  Note that for real numbers $x$ and $y$, we have
\[ |x-y|  =  |x|+|y| - 2\min(|x|,|y|). \]
Hence, we have
\begin{equation}\label{E gross}
\e |W - V_\alpha|  =   \e X_\alpha(\pi) + \e \sum_{\beta \ne \alpha} |X_\beta(\pi) - X_\beta(\pi')| - A_1 - A_2,
\end{equation}
where 
\[ A_2 = 2\, \e \min\left( \sum_\beta (X_\beta(\pi) - X_\beta(\pi'))^-,\ \sum_\beta (X_\beta(\pi) - X_\beta(\pi'))^+\right). \]
\end{proof}

We now address the calculation for the gross number of changes.  The possible changes in attacking pairs come from those rooks that were in attacking configurations before the coupling (and were moved out of attacking), and those in attacking configurations occurring after the coupling (i.e., were moved into attacking).  

First, we define $\Gamma_\alpha := \{(a',b'): |\{a',b'\} \cap \{a,b\}| = 1\}$ to be the set of all pairs of rooks that share exactly one rook with $\alpha$.  For $\beta \notin \Gamma_\alpha \cup \alpha$, we have $X_\beta(\pi) - X_\beta(\pi') = 0$.  For $\beta \in \Gamma_\alpha$, we have
\begin{align*}
 \e |X_\beta(\pi) - X_\beta(\pi')| & = \e X_\beta(\pi) + \e X_\beta(\pi') - 2\, \e \min(X_\beta(\pi), X_\beta(\pi')) \\
    & \leq \e X_\beta(\pi) + \e X_\beta(\pi') \\
    & = R_\beta(\pi) + C_\beta(\pi) + R_\beta(\pi') + C_\beta(\pi').
\end{align*}
Let $A_3 := 2\,\e \min(X_\beta(\pi), X_\beta(\pi'))$ denote the error by ignoring cancellations that may occur due to rooks in attacking positions both before and after the coupling.  

We have
 \begin{align}
\nonumber \e |W-V_\alpha|  = \e X_\alpha(\pi) &+ \sum_{\beta\in \Gamma_\alpha}( \e X_\beta(\pi)  + \e X_\beta(\pi')  ) - A_1 - A_2 - A_3 \\
\label{final form} =  \e X_\alpha(\pi) &+ \sum_{\beta\in \Gamma_\alpha} ( \e R_\beta(\pi) + \e C_\beta(\pi)  + \e R_\beta(\pi') + \e C_\beta(\pi')) - A_1 - A_2 - A_3 \\
\nonumber  \leq \e X_\alpha(\pi) &+ \sum_{\beta\in \Gamma_\alpha} ( \e R_\beta(\pi) + \e C_\beta(\pi)  + \e R_\beta(\pi') + \e C_\beta(\pi')).
\end{align}

Define $R^+ := \sum_{\beta\in\Gamma_\alpha}\e R_\beta(\pi)$ to be the expected number of pairs of rooks occupying the same row as one of the rooks in $\alpha$.  Similarly, define $C^+ := \sum_{\beta\in\Gamma_\alpha}\e C_\beta(\pi)$ to be the expected number of pairs of rooks occupying the same column as one of the rooks in $\alpha$.  We have
\begin{equation} \label{ER plus}
R^+ = \sum_{\beta\in\Gamma_\alpha} r_\beta = 2(k-2) r_\alpha,
\end{equation}
and
\begin{equation} \label{EC plus}
 C^+ = \sum_{\beta\in\Gamma_\alpha} c_\beta = 2(k-2) c_\alpha.
\end{equation}

The terms $\e R_\beta(\pi')$ and $\e C_\beta(\pi')$, respectively, depend on the particular locations $I$ and $J$ in which the coupled rooks are located.  We need to condition on which case from the coupling we are in.  For example, cases~1,~2a,~3a induce no noticeable changes in rook positions.  Cases~2b,~3b only change the position of one rook, and case~3c changes the positions of two rooks.  

We now consider case~3c.  
The quantity $\e R_\beta(\pi')$ is the expected number of pairs of rooks in the same row, conditioned on the pair of rooks $\alpha = \{a,b\}$ in a particular attacking configuration.  There are two cases to consider when the attacking rooks are coupled to be attacking on the same row: 
\begin{itemize}
\item[(A)] Any other rook $s$ that attacks in the same row will now contribute 2 additional attacks, one for each pair $\{s, a\}$ and $\{s, b\}$.  See Figure \ref{R1 diagram}.
\item[(B)]An attacking rook pair in the same column as either of the two coupled rooks will contribute 1 attack.  See Figure \ref{R2 diagram}.
\end{itemize}

We consider each case starting with Case~(A):  conditional on the event $\{X_\alpha=1\}$, the probability that the pair $\alpha = \{a,b\}$ is attacking in row $r$ --- that is, the probability that random variables $I$ and $J$ have first coordinate $r$ ---  is $\binom{r}{2}/(2\binom{n}{3})$.\footnote{Instead of letting pairs of rooks lie in any of the $N$ possible pairs of positions, they can only appear pairwise in a row, proportional to the number of pairs of positions in that row.  Similarly with columns.  This implies $2 C \sum_{r=2}^{n-1} \binom{r}{2} = 1$, hence $C = (2\binom{n}{3})^{-1}$.}  Once this row is determined, there are $r-2$ remaining spaces for another rook to appear in the same row $r$.
Let $R_A^-$ denote the contribution from this case.  We have
\begin{equation}\label{R1 Equation}
R_A^- = 2(k-2) \sum_{r=2}^{n-1} \frac{\binom{r}{2}}{2\binom{n}{3}} \frac{r-2}{\left(N-2\right)} = \frac{3 \binom{n}{4}(k-2)}{\binom{n}{3}\left(N-2\right)}.
\end{equation}
We define $C_A^-$ analogously replacing rows with columns, and by symmetry we have $C_A^- = R_A^-$.  

In Case~(B) we have to sum over all possible positions of $i$ and $j$ within each row, i.e., the set of all possible column coordinates of $I$ and $J$ conditional on having row coordinate $r$.  
Let $R_B^-$ denote the contribution from this case.  We have

\begin{align}\label{R2 Equation}
R_B^- & =   \frac{k-2}{2\binom{n}{3}} \sum_{r=2}^{n-1} \sum_{c_1 = 1}^{r-1} \sum_{c_2 = c_1 + 1}^{r} \left[ \frac{n-1-c_1}{\left(N-2\right)} + \frac{n-1-c_2}{\left(N-2\right)}\right]  \\
\nonumber          & =  \frac{k-2}{2\binom{n}{3}\left(N-2\right)}\left(2(n-1) \binom{n}{3} - \binom{n}{3}\frac{n+1}{4} - 2 \binom{n+1}{4} \right) \\
\nonumber & =  \frac{k-2}{2\left(N-2\right)}\left(2(n-1) - \frac{n+1}{4} - \frac{n+1}{2} \right) \\
\nonumber & =  \frac{(k-2)}{\left(N-2\right)} \left(\frac{5n-11}{8}\right).
 \end{align}
We define $C_B^-$ analogously by exchanging rows with columns, and by symmetry we have $C_B^- = R_B^-$.

\begin{figure}[h]
\[
\young(:::::::::\ ,::::::::\ \ ,:::::::\ \ \ ,::::::\ \ \ \ ,:::::\ \ \ \ \ ,::::\ \ \ \ \ \ ,:::\ \ \ \ \ \ \ ,::\ast\ast\ast a\ast b\ast\ast ,:\ \ \ \ \ \ \ \ \ ,\ \ \ \ \ \ \ \ \ \  )
\]
\caption{The locations of potential row-attacks ($\ast$) relative to row-coupled rooks $a$ and $b$.}
\label{R1 diagram}
\end{figure}

\begin{figure}[h]
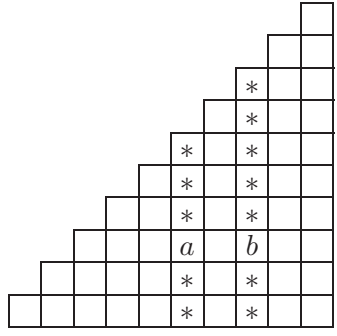

\[
\young(:::::::::\ ,::::::::\ \ ,:::::::\ast \ \ ,::::::\ \ast \ \ ,:::::\ast \ \ast \ \ ,::::\ \ast \ \ast \ \ ,:::\ \ \ast \ \ast \ \ ,::\ \ \ a\ b\ \ ,:\ \ \ \ \ast \ \ast\ \ ,\ \ \ \ \ \ast \ \ast \ \  )
\]
\caption{The locations of potential column-attacks ($\ast$) relative to row-coupled rooks $a$ and $b$.}
\label{R2 diagram}
\end{figure}

By combining Equations~\eqref{R1 Equation}~and~\eqref{R2 Equation}, we have
\begin{equation}\label{RR1 Equation}
\sum_{\beta \in \Gamma_\alpha} \e R_\beta(\pi') = R_A^- + R_B^- =  (k-2)\frac{3 \binom{n}{4}}{\binom{n}{3}\left(N-2\right)} + \frac{(k-2)}{\left(N-2\right)} \left(\frac{5n-11}{8}\right).
\end{equation}

In order to combine Equations~\eqref{final form},~\eqref{ER plus},~\eqref{EC plus},~and~\eqref{RR1 Equation}, we must first separate the coupling into its various cases.  In cases 1, $2a$, or $3a$, we have $|W-V_\alpha| = 1$; this event occurs with probability $T$ given by Equation~\eqref{cases 123}.  
For cases 2b and 3b, instead of conditioning on one rook changing, we instead assume that two rooks always change, as in case 3c, and let the difference contribute to the error.  
Define $A_4 := \P(\text{coupling is case 2b or 3b})$, then we have
\[ \P(\text{case 3c}) = 1 - T - A_4, \]
where $A_4 = O(k/n^2)$ and $T = O(k^2/n^4).$  We next define $A_5$ as the difference in the conditional expectations of the number of changes under the two different events; that is, 
\[ A_5 :=  \e\left( |W-V_\alpha|\ | \text{ case 3c} \right) - \e\left( |W-V_\alpha|\ | \text{ case 2b, 3b}\right),  \]
and note that $A_5 \geq 0$.  

We have

\begin{align}
\nonumber \e |W-V_\alpha| & = \ \ \, \e\left( |W-V_\alpha|\  | \text{ case 1, 2a, or 3a} \right) \ \P(  \text{case 1, 2a, or 3a} )  \\
\nonumber & \ \ \ + \e\left( |W-V_\alpha|\ | \text{ case 2b, 3b} \right) \P(  \text{case 2b, 3b} )  \\
\nonumber & \ \ \ + \e\left( |W-V_\alpha|\ | \text{ case 3c} \right) \P(  \text{case 3c} )  \\
\nonumber & = \ \ \ (r_\alpha + c_\alpha) + T +  \\
\nonumber &  \ \ \ + A_4\left( \e\left( |W-V_\alpha|\ | \text{ case 2b, 3b} - \e\left( |W-V_\alpha|\ | \text{ case 3c} \right)\right)\right) \\
\nonumber &  \ \ \ + (1-A_4-T)\left( R^+ + R_A^- + R_B^- + C^+ + C_A^- + C_B^- - A_1 - A_2 - A_3\right) \\
\nonumber & = \ \ \ r_\alpha + c_\alpha + T + (1-T)(R^+ + R_A^- + R_B^- + C^+ + C_A^- + C_B^-) \\
\nonumber &  \ \ \ -A_4 A_5 \\
\nonumber & \ \ \ - (1-A_4-T)(A_1 + A_2 + A_3)  \\
\nonumber & \ \ \ - A_4\left(R^+ + R_A^- + R_B^- + C^+ + C_A^- + C_B^-\right) \\
\label{proof error} & \leq r_\alpha + c_\alpha + T + (1-T)(R^+ + R_A^- + R_B^- + C^+ + C_A^- + C_B^-) \\
\nonumber & =  \ \frac{2\binom{n}{3}}{\binom{N}{2}} + \frac{13-12 k+3 k^2}{\binom{N}{2}} + \\
\nonumber & \ \ \  \left(1 -\frac{13-12 k+3 k^2}{\binom{N}{2}}\right) \left(\frac{8 (k-2)\binom{n}{3}}{\binom{N}{2}} + (k-2)\frac{6 \binom{n}{4}}{\binom{n}{3}\left(N-2\right)} + \frac{(k-2)}{\left(N-2\right)} \left(\frac{5n-11}{4}\right) \right). \\
\end{align}

Next, we observe that our quantities do not depend on the particular rooks chosen for the coupling, hence for $\alpha' = \{1,2\}$, we have
\begin{align} \label{together}
\sum_{\alpha} (r_\alpha + c_\alpha)  \e |W-V_\alpha| & = \binom{k}{2} (r_{\alpha'}+c_{\alpha'})  \e |W-V_{\alpha'}| = 2\, \mu\,  \e |W-V_{\alpha'}|.  
  \end{align}
We now define
\begin{align*}
d_1 & := c_{\alpha'} +  T + (1-T)(C^+ + 2 C_A^-), \\
d_2 & := r_{\alpha'} + c_{\alpha'} + T + (1-T)(R^+ + R_A^- + R_B^- + C^+ + C_A^- + C_B^-).
\end{align*}
  Finally, using \eqref{ChenSteinBound}, for Stirling numbers of the second kind, the total variation distance is then bounded by
\begin{align}
\nonumber d_{TV}(W,P) & \leq  \min(1, (2\mu)^{-1}) \sum_\alpha (r_\alpha+c_\alpha) \e|W-V_\alpha| \\
\label{dtv2} & \leq \min(2\mu, 1)\, d_2,
\end{align}
and for Stirling numbers of the first kind, the total variation distance is bounded by
\begin{align}
\nonumber d_{TV}(W,P) & \leq  \min(1, \mu^{-1}) \sum_\alpha r_\alpha \e|W-V_\alpha| \\
\label{dtv1} & \leq \min(\mu, 1)\, d_1.
\end{align}
This completes the proof of Theorem~\ref{Stirling bound}.

\begin{remark}{\rm
The expression for $d_1$ is slightly different than the expression for $d_2$ since rooks can only attack column-wise.  
We obviously ignore contributions from row--attacking rooks.  We also ignore the contribution from $C_B^-$, which counts the number of row-wise conflicts affected by the coupling.  

In addition, there are half as many possible ways for rooks to be in attacking configuration, so the normalization factor $2 \binom{n}{3}$ becomes just $\binom{n}{3}$ in Equation~\eqref{R1 Equation}, hence the $2 C_A^-$ instead of $C_A^-$.
}\end{remark}

\begin{remark}\label{remark error}{\rm
The total error introduced to $\e |W - V_\alpha|$ by Equation~\eqref{proof error} is given by 
\[ A_4 A_5 + (1-A_4-T)(A_1 + A_2 + A_3) + (A_4+T) \left(R^+ + R_A^- + R_B^- + C^+ + C_A^- + C_B^-\right).\]
This expression is \emph{positive} and $O(k/n)$.  
Summing over all $\alpha$, and multiplying by $\e X_\alpha$, as in Equation~\eqref{together}, the total error introduced is $O(k^3/n^2)$, which is the same order the total variation distance upper bounds in Equations~\eqref{dtv2}~and~\eqref{dtv1}.  
Thus, even if cancellations were taken into account, the upper bound can be improved by at most a constant factor.  
}\end{remark}

\section{Numerical Calculations for small $n$ and $k$}

As a check for our bound in Theorem~\ref{Stirling bound}, we demonstrate the effectiveness for small values of $n$ and $k$, which can be checked on a computer via simulation.  

First we calculate the \emph{exact} bound $\e |W - V_\alpha|$ in the case of $k=2$, $n \geq 3$.  Call the two rooks $a$ and $b$.  There are exactly $\binom{N}{2}$ ways of placing these two rooks on the board, and $2\binom{n}{3}$ ways of placing them on the board such that they are attacking.  Thus, $|W - V_\alpha|$ is precisely the indicator random variable that the two random placements of rooks were in attacking configuration before the coupling.  This has probability 
\[ \e |W - V_\alpha| = P(\text{$a$ and $b$ attacking before coupling}) = \frac{2\binom{n}{3}}{\binom{N}{2}}. \]
Compare this to 
\[D_{n,2}^{(2)} = \frac{2\binom{n}{3}}{\binom{N}{2}}+\frac{1}{\binom{N}{2}}.
\]
  The added error comes from the probability that the locations of $a$ and $b$ stay the same before and after the coupling, which was deemed an acceptable loss of error for the sake of easier analysis in the proof of Theorem~\ref{Stirling bound}.

The case when $k=3$, $n\geq 3$, can also be worked out exactly.  In this case $|W~-~V_\alpha|~\in~\{0,1,2\}$, and one can enumerate through the various scenarios for attacking rook pairs before and after the coupling.

A table for $D_{n,k}^{(2)}$ is provided in Figure~\ref{fig:table}, with each row corresponding to the value of $n=11,\ldots,30,$ and each column corresponding to $k=3,4$.

\begin{figure}[h]
\[
\left(
\begin{array}{c|cc}
D_{n,k}^{(2)} & k=3 & k=4 \\ \hline
n=11 &  1. & 1. \\
n=12 &  0.879907 & 1. \\
n=13 &  0.759265 & 1. \\
n=14 &  0.661861 & 1. \\
n=15 &  0.582082 & 1. \\
n=16 &  0.515913 & 1. \\
n=17 &  0.460423 & 1. \\
n=18 &  0.413431 & 1. \\
n=19 &  0.373286 & 1. \\
n=20 &  0.338717 & 1. \\
n=21 &  0.308739 & 1. \\
n=22 &  0.282572 & 1. \\
n=23 &  0.259596 & 0.964415 \\
n=24 &  0.239314 & 0.889072 \\
n=25 &  0.221319 & 0.822227 \\
n=26 &  0.20528 & 0.762648 \\
n=27 & 0.190924 & 0.709319 \\
n=28 & 0.178023 & 0.661395 \\
n=29 &  0.166386 & 0.61817 \\
n=30 &  0.155855 & 0.579048
\end{array}
\right)
\]
\caption{Table of values of the absolute error in Poisson approximation for Stirling numbers of the second kind using $k=3, 4$, and $n=11,\ldots,30.$}
\label{fig:table}
\end{figure}

\newpage
\section{Remarks}
\label{remarks}
\subsection{An alternative coupling}
\label{independence coupling}
We now modify our scenario to demonstrate how alternative bounds can be obtained.  
Recall that our board $B$ is chosen to be the strictly lower triangular squares in a $n\times n$ board.
There are  $N = \binom{n}{2}$ squares on the board, and we assume that $k$ \emph{distinguishable} rooks are placed \emph{independently}, so that our sample space has size $N^k$ equally likely outcomes. 
A similarly defined coupling, which we call the \emph{independence coupling}, then simply moves two rooks, say 1 and 2,  into positions $I$ and $J$ without needing to swap them.  This coupling is much simpler than the one presented in the proof of Theorem~\ref{Stirling bound}.
The calculations for the total variation bound are similar, and we summarize them below.

Let $S_n(i) = \sum_{r=1}^{n-1} r^i$. 
For any two rooks, let $A(n)$ denote the total number of attacking positions.  We have 
\[ A(n) = \sum_{c=1}^{n-1} c^2 + \sum_{r=1}^{n-1}2\binom{r}{2} = \frac{1}{6}n(n-1)(4n-5). \]
Let $A_c(n)$ denote the number of attacking positions of two rooks when attacks are only column-wise, and let $A_r(n) = 2\binom{n}{3}$.  We have $A_c(n) = S_n(2) = n(n-1)(2n-1)/6$.
  
We define 
\[
\begin{array}{llll}
c_\alpha^{(I)} & := \frac{A_c(n)}{N^2}, &  r_\alpha^{(I)} & := \frac{A_r(n)}{N^2}, \\
C^{+,(I)} & :=  2(k-2)c_\alpha^{(I)}, & R^{+,(I)} & := 2(k-2)r_\alpha^{(I)} 
\end{array} \]
\begin{align*}
R_A^{-,(I)} & := 2(k-2) \sum_{r=1}^{n-1} \frac{r^2}{A(n)}\frac{r}{N}  = \frac{2(k-2) S_n(3)}{A(n)\, N} \\
C_A^{-,(I)} & := \frac{A(n)}{A_c(n)} R_A^{-,(I)}, \\
R_B^{-,(I)} & := \frac{k-2}{A(n)} \frac{1}{N} \sum_{r=1}^{n-1} \sum_{c_1=1}^{r} \sum_{c_2=c_1}^{r} (r+(n-c_1)+(n-c_2)) = \frac{2(k-2)(2n-1)\binom{n+1}{3}}{A(n)\, N}, \\
C_B^{-,(I)} & := \frac{A(n)}{A_c(n)} R_B^{-,(I)}, \\
 d_1^{(I)} & := c_\alpha^{(I)} + C^{+,(I)} + C_A^{-,(I)} + C_B^{-,(I)}, \\
 d_2^{(I)} & := r_\alpha^{(I)}+c_\alpha^{(I)} + 2(R^{+,(I)} + R_A^{-,(I)} + R_B^{-,(I)}). \\
\end{align*}
We then have an analogous theorem using the independence coupling.
\begin{theorem}
\label{Stirling bound independence}
Suppose $n\geq 3$ and $n \geq k \geq 2$.  
Define $N := \binom{n}{2}$ and 
\[
\mu_1^{(I)} := \binom{k}{2}\frac{n(n-1)(2n-1)}{6\, N^2},\ \ \ \ \ \ \mu_2^{(I)} := \binom{k}{2}\frac{n(n-1)(4n-5)}{6\, N^2},
\]
\[ d_1^{(I)} = \frac{n(n-1)(2n-1)}{6\, N^2} + 2(k-2)\frac{n(n-1)(2n-1)}{6\, N^2} +  \frac{3(k-2) n(n-1)}{(2n-1)\, N} + \frac{2(k-2)(n+1)}{N},
\]
\[ d_2^{(I)} = \frac{n(n-1)(4n-5)}{6\, N^2} + 2\left(4(k-2)\frac{n(n-1)(2n-1)}{6\, N^2} + \frac{3(k-2) n(n-1)}{(4n-5)\, N}  + \frac{2(k-2)(2n-1)(n+1)}{(4n-5)\, N} \right),
\]
\[ D_{n,k}^{(1),(I)} := \min\left(d_1^{(I)}, \mu^{(I)}\, d_1^{(I)}, 1\right), \]
\[ D_{n,k}^{(2),(I)} := \min\left(d_2^{(I)}, 2\, \mu^{(I)}\, d_2^{(I)}, 1\right). \]
Then we have
\[\frac{N^k}{k!} e^{-\mu_2^{(I)}}\left( 1 - e^{\mu_2^{(I)}} D_{n,k}^{(2),(I)}\right) \leq S(n, n-k) \leq \frac{N^k}{k!} e^{-\mu_2^{(I)}}\left( 1 + e^{\mu_2^{(I)}} D_{n,k}^{(2),(I)}\right).\]
\[\frac{N^k}{k!} e^{-\mu_1^{(I)}}\left( 1 - e^{\mu_1^{(I)}} D_{n,k}^{(1),(I)}\right) \leq |s(n, n-k)| \leq \frac{N^k}{k!} e^{-\mu_1^{(I)}}\left( 1 + e^{\mu_1^{(I)}} D_{n,k}^{(1),(I)}\right).\]
\end{theorem}

\subsection{An historical note on small values of $m$}\label{small m}

For values of $m$ fixed as $n\to\infty$, Jordan \cite{Jordan} is credited by Moser and Wyman \cite{MoserWyman1} for the asymptotic formula 
\[
|s(n,m)| \sim  \frac{(n-1)!}{(m-1)!}\left( \log n + \gamma\right)^{m-1}.
\]
Moser and Wyman \cite{MoserWyman1} extended this \emph{first-order} asymptotic formula to values of $m = o(\log n)$.  Wilf \cite{Wilf} extended this formula into an asymptotic \emph{series} including higher order terms valid for $m = O(1)$, and Hwang \cite{Hwang} extended Wilf's asymptotic series to $m~=~O(\log(n))$.  

\subsection{ Rook and File numbers}

We have taken our board $B$ to be the staircase board, i.e., the board with rows of size $n-1$, $n-2$, \ldots, $2, 1$.  One can generalize this to a board with rows of integer lengths $b = (b_1, b_2, \ldots)$, such that $b_1 \geq b_2 \geq \ldots b_\ell \geq 1$.  The sequence of positive integers is an integer partition, and so the board $B$ is called a Ferrers board.  The column sizes are given by the conjugate partition to $b$, which we denote by $b' = (b_1', b_2', \ldots)$, where $b_i' = \#\{j : b_j \geq i\}$.    
The rook number $r_B(k)$ is defined as the number of ways of placing $k$ non--attacking rooks on board $B$, and the file number $f_B(k)$ is defined as the number of ways of placing $k$ rooks on board $B$ such that attacks only occur column-wise.

In this setting, the permutation coupling is the same, and the analogous calculations for the error bound~\eqref{E gross} are represented using unsimplified sums.  

\begin{theorem}\label{rook theorem}
Given a Ferrers board $B$ with row lengths given by the integer partition $b = (b_1, b_2, \ldots)$, define $N = \sum_i b_i$, $L = \sum_i \binom{b_i}{2}$, $L' = \sum_i \binom{b_i'}{2}$, 
\[
\lambda := \frac{\binom{k}{2} L}{ \binom{N}{2}}, \qquad \lambda' = \frac{\binom{k}{2} L'}{ \binom{N}{2}}
\]
\[ s_0 := \frac{L}{\binom{N}{2}}, \qquad  s_0' = \frac{L'}{\binom{N}{2}}, \]
\[ s_1 := \frac{2 (k-2) L}{\binom{N}{2}}, \qquad  s_1' = \frac{2 (k-2) L'}{\binom{N}{2}}, \]
\[ s_2 :=  (k-2) \left[\sum_{i} \binom{b_i}{2} \frac{b_i-2}{\left(N-2\right)} \right]  \qquad  s_2' :=  (k-2) \left[
 \sum_{i} \binom{b_i'}{2} \frac{b_i'-2}{\left(N-2\right)}\right] \]
 \[ s_3 := (k-2) \sum_{i} \sum_{c_1 = 1}^{b_i-1} \sum_{c_2 = c_1 + 1}^{b_i} \left[ \frac{ b_{c_1}'-1}{\left(N-2\right)} + \frac{b_{c_2}'-1}{\left(N-2\right)}\right]  \] 
 \[ s_3' := (k-2) \sum_{i} \sum_{c_1 = 1}^{b_i'-1} \sum_{c_2 = c_1 + 1}^{b_i'} \left[ \frac{ b_{c_1}-1}{\left(M-2\right)} + \frac{b_{c_2}-1}{\left(M-2\right)}\right]  \]
\[ s_4 := \frac{13 -12k +3k^3}{\binom{N}{2}}, \]
\[ d_3 := s_0 + s_0' + s_1 + s_1' + \frac{s_2+s_2'+s_3+s_3'}{L+L'} + s_4,\]
\[ d_4 := s_0' + s_1' + \frac{s_2' + s_3'}{L'} + s_4,\]
and 
\[ R_{n,k} := \min\left(d_3, (\lambda+\lambda')d_3, 1\right), \qquad F_{n,k} := \min\left(d_4, \lambda'd_4, 1\right). \]

Then we have
\[\binom{N}{k} e^{-(\lambda+\lambda')}\left( 1 - e^{\lambda+\lambda'} R_{n,k}\right) \leq r_B(k) \leq \binom{N}{k} e^{-(\lambda+\lambda')}\left( 1 + e^{\lambda+\lambda'} R_{n,k}\right).\]
\[\binom{N}{k} e^{-\lambda'}\left( 1 - e^{\lambda'} F_{n,k}\right) \leq f_B(k) \leq \binom{N}{k} e^{-\lambda'}\left( 1 + e^{\lambda'} F_{n,k}\right).\]
\end{theorem}

It would be interesting to investigate conditions on $k$ and row lengths $b$ such that the Poisson approximation holds.

\subsection{Generalizations and Extensions}
There are extensions of Stirling numbers to complex-valued arguments \cite{FlajoletProdinger}.  We are not aware of any combinatorial arguments that would correspond to placements of rooks on a board, so we do not pursue this idea further.

One can also generalize rook and file numbers to include non-uniform weights on the squares, see for example \cite{Stirlinggeneralized} and the references therein.  One can still apply Poisson approximation in these cases.  In fact, the coupling is identical, and the expressions for the bounds are weighted sums.

\section{Acknowledgements}

We would like to thank Amir Dembo, Jim Haglund, Igor Pak, and Bruce Rothschild for helpful discussions, and an anonymous referee for constructive criticisms.

\bibliographystyle{acm}
\bibliography{StirlingBib}

\end{document}